\documentclass[a4paper]{amsart}
\usepackage{amsmath, amsthm, amsfonts, mathtools, tikz-cd, amssymb}
\usepackage{enumitem}
\usepackage{hyperref}

\title[Contractible Rips complexes of groups via metric gluings]{Contractible Rips complexes of groups \\ via metric gluings}
\date{August 25, 2026}


\author[K.~Li]{Kevin Li}
\address{Institut f\"ur Mathematik, Freie Universit\"at Berlin, 14195 Berlin, Germany}
\email{kevin.li@fu-berlin.de}

\author[L.J.~S\'anchez Salda\~na]{Luis Jorge S\'anchez Salda\~na}
\address{Departamento de Matem\'aticas, Facultad de Ciencias, Universidad Nacional Aut\'onoma de M\'exico, 04510 Ciudad de M\'exico, Mexico}
\email{luisjorge@ciencias.unam.mx}

\theoremstyle{definition}
\newtheorem{defn}{Definition}[section]
\newtheorem{ex}[defn]{Example}
\theoremstyle{plain}
\newtheorem{lem}[defn]{Lemma}
\newtheorem{prop}[defn]{Proposition}
\newtheorem{thm}[defn]{Theorem}
\newtheorem{cor}[defn]{Corollary}

\newcommand{\enum}{\rm{(\roman*)}}

\newcommand{\IN}{\ensuremath{\mathbb{N}}}
\newcommand{\IZ}{\ensuremath{\mathbb{Z}}}
\newcommand{\IR}{\ensuremath{\mathbb{R}}}

\newcommand{\sfF}{\ensuremath{\mathsf{F}}}
\newcommand{\sfR}{\ensuremath{\mathsf{R}}}

\DeclareMathOperator{\diam}{diam}
\DeclareMathOperator{\CAT}{CAT(0)}

\begin{document}

\begin{abstract}
	We say that a finitely generated group equipped with a finite generating set is of type~$\sfR$ if the Rips complex is contractible for sufficiently large scales.
	We show that the class of groups of type~$\sfR$ is closed under taking graphs of groups with finite edge groups and, in particular, under taking free products.
	We prove that all two-dimensional right-angled Artin groups with the standard generating set are of type~$\sfR$.
	Furthermore, we observe that products of finitely generated free abelian groups and finite groups are of type~$\sfR$.
\end{abstract}

\maketitle

\section{Introduction}
Rips complexes (also called Vietoris--Rips complexes) are an important tool in both applied topology and geometric group theory.
We use gluing techniques that arose in the former area to approach questions in the latter.

Let~$X$ be a metric space and let~$r\in \IR_{\ge 0}$.
The \emph{Rips complex~$R_r(X)$ of~$X$ at scale~$r$} is the simplicial complex of finite subsets of~$X$ with diameter~$\le r$.
Let~$G$ be a finitely generated group and let~$S$ be a finite generating set of~$G$.
We equip~$G$ with the word metric~$d_S$ associated to~$S$, for~$g,h\in G$ with~$g\neq h$ given by
\[
	d_S(g,h)\coloneqq \min\bigl\{k\in \IN\mid g^{-1}h=s_1\cdots s_k \text{ for some } s_1,\ldots,s_k\in S\cup S^{-1}\bigr\}.
\]
We denote the Rips complex of~$(G,d_S)$ by~$R_r(G,S)$.

\begin{defn}
\label{defn:type_R}
	A \emph{marked group} is a pair~$(G,S)$, where~$G$ is a finitely generated group and~$S$ is a finite generating set of~$G$.
	We say that a marked group~$(G,S)$ is \emph{of type~$\sfR$} if there exists~$r_0\in \IR_{\ge 0}$ such that for all~$r\in \IR$ with~$r\ge r_0$, the Rips complex~$R_r(G,S)$ is contractible.
\end{defn}

A priori, being of type~$\sfR$ may depend on the finite generating set.
However, we are not aware of a group~$G$ admitting two finite generating sets~$S_1$ and~$S_2$ such that~$(G,S_1)$ is of type~$\sfR$ and~$(G,S_2)$ is not of type~$\sfR$.
Finite groups with any generating set have finite diameter and are hence of type~$\sfR$.
It is a classical result of Rips that hyperbolic groups with any finite generating set are of type~$\sfR$~\cite[Proposition~III.$\Gamma$.3.23]{Bridson-Haefliger}. 
Zaremsky~\cite{Zaremsky22} asked whether finitely generated free abelian groups with the standard generating set are of type~$\sfR$.
This was proved recently by Virk~\cite{Virk}.

\begin{thm}[{\cite[Theorem~5.1]{Virk}}]
	Let~$n\in \IN$ and let~$\IZ^n$ be equipped with the standard generating set~$S$.
	Then~$(\IZ^n,S)$ is of type~$\sfR$.
\end{thm} 

It is the subject of ongoing research to establish the optimal (i.e., lowest) threshold~$r_0$ such that~$R_r(\IZ^n,S)$ is contractible for all~$r\ge r_0$.
Conjecturally, $R_r(\IZ^n,S)$ is contractible for all~$r\ge n$~\cite{Virk,Zaremsky26}.
At the time of writing, this is know to be true only for~$n\le 5$~\cite{Virk,GSS, Wang-Zhou}.
We refer to~\cite{GSS, Wang-Zhou} for the latest account on these quantitative aspects.
We focus more on the qualitative property type~$\sfR$.

Our motivation to study the property type~$\sfR$ comes from viewing it as a strong finiteness property of groups.
A group~$G$ is \emph{of type~$\sfF_\infty$} if there exists a CW-model for the classifying space~$BG$ that is of finite type. If there exists a contractible cocompact $G$-CW-complex with finite cell stabilisers, then~$G$ is of type~$\sfF_\infty$.
The barycentric subdivision of~$R_r(G,S)$ is a cocompact $G$-CW-complex with finite cell stabilisers.
Hence, if~$(G,S)$ is of type~$\sfR$, then~$G$ is of type~$\sfF_\infty$.
Conversely, it is not known if every group~$G$ of type~$\sfF_\infty$ that admits a finite-dimensional model for~$BG$ has a finite generating set~$S$ such that~$(G,S)$ is of type~$\sfR$~\cite[Question~4.6]{Zaremsky_problems}.

We prove several results for groups of type~$\sfR$ that are well-known for type~$\sfF_\infty$, thus providing more evidence for a strong connection between the two properties.
Our main tool is a gluing criterion for Rips complexes of metric gluings (Theorem~\ref{thm:gluing}) by Adamaszek et al.~\cite{Adamaszek_etal}.

\subsection*{Graphs of groups with finite edge groups.}

We prove combination theorems for amalgamated products and HNN-extensions over finite groups.

\begin{thm}
\label{thm:amalg}
	Let~$A\ast_C B$ be an amalgamated product, where~$A$ and~$B$ are finitely generated groups and~$C$ is a finite subgroup of both~$A$ and~$B$.
	Let~$S_C$ be the set of all non-trivial elements of~$C$.
	Let~$S_A$ and~$S_B$ be finite generating sets of~$A$ and~$B$, respectively, that contain~$S_C$.
	If the marked groups~$(A,S_A)$ and~$(B,S_B)$ are of type~$\sfR$, then the amalgamated product~$(A\ast_C B,S_A\cup S_B)$ is of type~$\sfR$.
\end{thm}

In the special case when~$C$ is the trivial group, Theorem~\ref{thm:amalg} reduces to free products (Corollary~\ref{cor:free_prod}).

\begin{thm}
\label{thm:HNN}
	Let~$A\ast_{C,\varphi}$ be an HNN-extension, where~$A$ is a finitely generated group, $C$ is a finite subgroup of~$A$, and~$\varphi\colon C\to A$ is a monomorphism.
	Let~$S_C$ be the set of all non-trivial elements of~$C$.
	Let~$S_A$ be a finite generating set of~$A$ containing~$S_C$ and~$\varphi(S_C)$.
	If the marked group~$(A,S_A)$ is of type~$\sfR$, then the HNN-extension~$(A\ast_{C,\varphi},S_A\cup\{t\})$ is of type~$\sfR$, where~$t$ is the stable letter.
\end{thm}

Together, Theorem~\ref{thm:amalg} and Theorem~\ref{thm:HNN} show that the fundamental group of a finite graph of groups~\cite{Serre} with vertex groups of type~$\sfR$ and finite edge groups equipped with a suitable finite generating set is of type~$\sfR$.

The proof of Theorem~\ref{thm:amalg} (and similarly of Theorem~\ref{thm:HNN}) relies on decomposing the Cayley graph of the amalgamated product as a successive gluing of copies of the Cayley graphs of~$(A,S_A)$ and~$(B,S_B)$ along copies of the Cayley graph of~$(C,S_C)$.
We need the group~$C$ to be finite because the gluing criterion for Rips complexes (Corollary~\ref{cor:gluing_star}) requires a gluing locus of small diameter.
We do not know if in Theorem~\ref{thm:amalg} and Theorem~\ref{thm:HNN} it is enough to assume that the edge group~$(C,S_C)$ is of type~$\sfR$.

\subsection*{Two-dimensional right-angled Artin groups}
As mentioned above, finitely generated free groups and finitely generated free abelian groups with the standard generating sets are of type~$\sfR$.
A well-studied class of groups that interpolates between these two extremes is that of right-angled Artin groups~\cite{Charney_survey,Koberda_survey}.

Let~$\Gamma$ be a finite simplicial graph.
The \emph{right-angled Artin group~$A_\Gamma$} has as generators the vertices~$\Gamma^{(0)}$ of~$\Gamma$ subject to the relations that two generators commute if the corresponding vertices span an edge in~$\Gamma$.
It is a natural question if all right-angled Artin groups~$(A_\Gamma,\Gamma^{(0)})$ are of type~$\sfR$~\cite[Question~4.5]{Zaremsky_problems}.
We prove that this is true for all two-dimensional right-angled Artin groups (i.e., those that admit a CW-model for~$B(A_\Gamma)$ of dimension~$\le 2$).

\begin{thm}
\label{thm:RAAG}
	Let~$\Gamma$ be a finite simplicial graph without triangles.
	Then the right-angled Artin group~$(A_\Gamma,\Gamma^{(0)})$ is of type~$\sfR$.
	More precisely, for every~$r\ge 2$, the Rips complex~$R_r(A_\Gamma,\Gamma^{(0)})$ is contractible.
\end{thm}

In the special case when~$\Gamma$ is complete bipartite, the group~$A_\Gamma$ is a direct product of two free groups (Example~\ref{ex:bipartite}).
The proof of Theorem~\ref{thm:RAAG} relies on decomposing a CAT(0) square complex as a successive gluing of squares.
Our method is similar to that of Virk for~$\IZ^2$ via domination~\cite[Theorem~3.2]{Virk}.
In higher dimensions, the gluing criterion for Rips complexes (Theorem~\ref{thm:gluing}) becomes more complicated to verify and the combinatorics of the gluings intractable.

It was proved independently by Hulbert--Zaremsky~\cite{Hulbert-Zaremsky} via Morse theoretic methods that the Rips complex~$R_2(A_\Gamma,\Gamma^{(0)})$ at scale~$2$ is contractible for all two-dimensional right-angled Artin groups~$A_\Gamma$.

We point out that for a different choice of generating set~$S$, it is known that all right-angled Artin groups~$(A_\Gamma,S)$ are of type~$\sfR$~\cite[Lemma~5.20 and Theorem~4.1]{CCGHO25}.
Here~$S$ is the set of all products of generators in~$\Gamma^{(0)}$ corresponding to vertices spanning a complete subgraph in~$\Gamma$.

\subsection*{Products of free abelian groups and finite groups}
It is a natural question when type~$\sfR$ is compatible with products (and, more generally, with extensions).
The example of~$\IZ^n$ shows that this question is delicate in general.
As a first step, one may ask if type~$\sfR$ is inherited under taking products with finite groups.
We observe that this is true for finitely generated free abelian groups, based on a Morse theoretic criterion (Theorem~\ref{thm:Morse}) by Zaremsky~\cite{Zaremsky26}.

\begin{thm}
\label{thm:prod_finite}
	Let~$n\in \IN$ and let~$\IZ^n$ be equipped with the standard generating set~$S$.
	Let~$F$ be a finite group equipped with any generating set~$T$.
	Then the direct product~$(\IZ^n\times F,S\cup T)$ is of type~$\sfR$.
\end{thm}

Theorem~\ref{thm:prod_finite} applies to all finitely generated abelian groups.

\subsection*{Acknowledgements}
We thank Macarena Arenas for explanations about $\CAT$ cube complexes and Matthew Zaremsky for helpful discussions.
The second author is grateful for the financial support of the DGAPA-UNAM grant PAPIIT~IN102426.

\section{Rips complexes of metric gluings}

In this section, we recall a criterion for the compatibility of Rips complexes with metric gluings (Theorem~\ref{thm:gluing}) by Adamaszek et al.~\cite{Adamaszek_etal}.

\begin{lem}
\label{lem:exhaustion}
	Let~$X$ be a metric space and let $X_0\subset X_1\subset \cdots\subset X$ be isometric subspaces of~$X$ with~$\bigcup_{i\in \IN} X_i=X$.
	Let~$r\in \IR_{\ge 0}$.
	Suppose that for all~$i\in \IN$, the Rips complex~$R_r(X_i)$ is contractible.
	Then~$R_r(X)$ is contractible.
	\begin{proof}
		By the Whitehead theorem, it suffices to show that all homotopy groups of~$R_r(X)$ are trivial.
		Since spheres are compact, every continuous map from a sphere to~$R_r(X)$ factors through a finite subcomplex of~$R_r(X)$ and hence through~$R_r(X_i)$ for some~$i\in \IN$, which is contractible by assumption.
	\end{proof}
\end{lem}

For a simplicial complex~$K$ with vertex set~$V$ and a subset~$V_1$ of~$V$, we denote by~$K[V_1]$ the full subcomplex (also called the induced subcomplex) of~$K$ on the vertex set~$V_1$.

\begin{lem}[{\cite[Corollary~1]{Adamaszek_etal}}]
\label{lem:Whitehead}
	Let~$K$ be a simplicial complex with vertex set~$V$ and let~$L$ be a subcomplex of~$K$ with the same vertex set~$V$.
	Suppose that for every finite subset~$V_0\subset V$, there exists a subset~$V_1\subset V$ with~$V_0\subset V_1$ such that the inclusion~$L[V_1]\to K[V_1]$ is a homotopy equivalence.
	Then the inclusion~$L\to K$ is a homotopy equivalence.
\end{lem}

Lemma~\ref{lem:Whitehead} is also a consequence of the Whitehead theorem.
In~\cite[Corollary~1]{Adamaszek_etal}, the subset~$V_1\subset V$ is required to be finite but this is not needed in the proof.

\begin{defn}
	Let~$(X,d_X)$ and~$(Y,d_Y)$ be metric spaces and let~$A$ be a closed isometric subspace of both~$X$ and~$Y$.
	The \emph{metric gluing of~$X$ and~$Y$ along~$A$} is the gluing of sets~$X\cup_A Y$ equipped with the metric~$d\colon (X\cup_A Y)\times (X\cup_A Y)\to \IR$, for~$v,w\in X\cup_A Y$ given by
	\[
		d(v,w)\coloneqq \begin{cases}
			d_X(v,w) & \text{if } v,w\in X; \\
			d_Y(v,w) & \text{if } v,w\in Y; \\
			\inf\{d_X(v,a)+d_Y(a,w)\mid a\in A\} & \text{if } v\in X, w\in Y.
		\end{cases}
	\]
	The function~$d$ is indeed a metric on~$X\cup_A Y$ since~$A$ is closed in both~$X$ and~$Y$~\cite[Lemma~I.5.24]{Bridson-Haefliger}. 
\end{defn}

The isometric inclusions of~$X$, $Y$, and~$A$ into~$X\cup_A Y$ induce inclusions of Rips complexes. 
These yield an inclusion of the simplicial gluing of Rips complexes into the Rips complex of the metric gluing
\[
	\iota\colon R_r(X)\cup_{R_r(A)} R_r(Y)\to R_r(X\cup_A Y).
\]
The following result of Adamaszek et al.~\cite{Adamaszek_etal} is a criterion ensuring that the map~$\iota$ is a homotopy equivalence.
Similar gluing criteria were obtained in~\cite{CJST,Virk}.

\begin{thm}[{\cite[Theorem~1]{Adamaszek_etal}}]
\label{thm:gluing}
	Let~$X\cup_A Y$ be a metric gluing, where~$A$ is a proper metric space, and let~$r\in \IR_{\ge 0}$.
	Suppose that for all non-empty finite subsets~$U_X\subset X\setminus A$ and~$U_Y\subset Y\setminus A$ with~$\diam(U_X\cup U_Y)\le r$, there exists a unique maximal non-empty finite subset~$\sigma\subset A$ such that~$\diam(U_X\cup \sigma\cup U_Y)\le r$.
	Then the inclusion
	\[
		R_r(X)\cup_{R_r(A)} R_r(Y)\to R_r(X\cup_A Y)
	\]
	is a homotopy equivalence.
	\begin{proof}
		(Sketch).
		For the proof in the case that~$X$ and~$Y$, and hence also~$A$, are finite metric spaces, we refer to~\cite[Theorem~1]{Adamaszek_etal}.
		The argument is a sequence of simplicial collapses, where the unique maximality of~$\sigma\subset A$ is used to obtain an order in which to perform the simplicial collapses. 
		We provide a few more details on the reduction from infinite metric spaces to finite metric spaces.
		
		We will apply Lemma~\ref{lem:Whitehead} to the simplicial complex~$K=R_r(X\cup_A Y)$ and the subcomplex~$L=R_r(X)\cup_{R_r(A)} R_r(Y)$.
		Let~$V_0$ be a finite subset of~$X\cup_A Y$.
		By assumption, for all non-empty finite subsets~$U_X\subset V_0\cap (X\setminus A)$ and~$U_Y\subset V_0\cap (Y\setminus A)$ with~$\diam(U_X\cup U_Y)\le r$, there exists a unique maximal non-empty finite subset~$\sigma=\sigma(U_X,U_Y)\subset A$ such that~$\diam(U_X\cup \sigma\cup U_Y)\le r$.
		Moreover, since~$A$ is proper, for all~$x\in V_0\cap X$ and~$y\in V_0\cap Y$, there exists~$a=a(x,y)\in A$ with~$d(x,y)=d_X(x,a)+d_Y(a,y)$~\cite[Lemma~I.5.24]{Bridson-Haefliger}.	
		Let~$V_1\subset X\cup_A Y$ be the union of~$V_0$, all the subsets~$(\sigma(U_X,U_Y))_{U_X,U_Y}$ as above, and all the elements~$(a(x,y))_{x,y}$ as above.
		Then~$V_1$, equipped with the restricted metric from~$X\cup_A Y$, can be identified with the metric gluing~$(V_1\cap X)\cup_{V_1\cap A} (V_1\cap Y)$ and satisfies the assumption of Theorem~\ref{thm:gluing}.
		Since~$V_1$ is finite, Theorem~\ref{thm:gluing} for finite metric spaces yields that the inclusion
		\[
			L[V_1]=R_r(V_1\cap X)\cup_{R_r(V_1\cap A)} R_r(V_1\cap Y)\to K[V_1]=R_r(V_1)
		\]
		is a homotopy equivalence.
		Then Lemma~\ref{lem:Whitehead} implies that the inclusion~$L\to K$ is a homotopy equivalence.
	\end{proof}
\end{thm}

In Theorem~\ref{thm:gluing}, we added the assumptions that~$A$ is proper and that~$\sigma$ is finite which are not present in~\cite[Theorem~1]{Adamaszek_etal}.

We are mainly interested in metric spaces that are vertex sets of locally finite graphs (with edges of length~$1$) arising as Cayley graphs of finitely generated groups.
The assumption of Theorem~\ref{thm:gluing} is satisfied in the case of a star-shaped gluing locus.

\begin{cor}
\label{cor:gluing_star}
	Let~$X\cup_A Y$ be a metric gluing of vertex sets of locally finite graphs and let~$r\in \IR_{\ge 2}$.
	Suppose that there exists~$a_0\in A$ such that for all~$a\in A$, we have~$d(a,a_0)\le 1$.
	Then the inclusion
	\[
		R_r(X)\cup_{R_r(A)} R_r(Y)\to R_r(X\cup_A Y)
	\]
	is a homotopy equivalence.
	\begin{proof}
		We verify the assumption of Theorem~\ref{thm:gluing}.
		Let~$U_X\subset X\setminus A$ and~$U_Y\subset Y\setminus A$ be non-empty finite subsets with~$\diam(U_X\cup U_Y)\le r$.
		Take~$\sigma\subset A$ to be the set of all~$a\in A$ such that~$\diam(U_X\cup \{a\}\cup U_Y)\le r$.
		Since~$\diam(A)\le 2\le r$, we have $\diam(U_X\cup \sigma\cup U_Y)\le r$ and~$\sigma\subset A$ is unique maximal with this property.
		Since~$A$ has finite diameter and is locally finite, $A$ is finite and hence also~$\sigma$ is finite.
		It remains to show that~$\sigma$ is non-empty.
		Since~$A$ is proper, for all~$x\in U_X$ and~$y\in U_Y$, there exists~$a\in A$ with~$d(x,y)=d_X(x,a)+d_Y(a,y)$~\cite[Lemma~I.5.24]{Bridson-Haefliger}.
		In particular, $d_X(x,a)\le r-1$ and hence, for $a_0\in A$ given in the assumption of Corollary~\ref{cor:gluing_star},
		\[
			d_X(x,a_0)\le d_X(x,a)+d_X(a,a_0)\le r-1+1=r.
		\]
		Symmetrically, we obtain $d_Y(a_0,y)\le r$.
		Together, $\diam(U_X\cup \{a_0\}\cup U_Y)\le r$ and thus~$a_0\in \sigma$.
	\end{proof}
\end{cor}

One can apply Corollary~\ref{cor:gluing_star} inductively to finite successive gluings.
We will also want to consider infinite successive gluings of metric spaces~\cite[I.5.26]{Bridson-Haefliger}.

\begin{defn}
\label{defn:succ_gluing}
	Let~$(X_i)_{i\in \IN}$ be a sequence of metric spaces.
	For~$i=0,1,2,\ldots$, let~~$A_i$ be a closed isometric subspace of both~$Y_i$ and~$X_{i+1}$, where~$Y_0\coloneqq X_0$ and~$Y_{i+1}$ is the metric gluing~$Y_i\cup_{A_i} X_{i+1}$.
	The \emph{successive metric gluing~$\bigcup_{A_*} X_*$ of~$(X_i)_i$ along~$(A_i)_i$} is the set~$Y_\infty\coloneqq \operatorname{colim}_i Y_i$ equipped with the metric~$d\colon Y_\infty\times Y_\infty\to \IR$, for~$v,w\in Y_\infty$ given by $d(v,w)\coloneqq d_{Y_{n}}(v,w)$ if~$v,w\in Y_n$.
	The function~$d$ is well-defined since~$Y_i$ is an isometric subspace of~$Y_{i+1}$ for every~$i\in \IN$.
\end{defn}

\begin{prop}
\label{prop:succ_gluing_star}
	Let~$\bigcup_{A_*} X_*$ be a successive gluing of vertex sets of locally finite graphs and let~$r\in \IR_{\ge 2}$.
	Suppose that the following hold for every~$i\in \IN$:
	\begin{enumerate}[label=\enum]
		\item There exists~$a_i\in A_i$ such that for all~$a\in A_i$, we have~$d(a,a_i)\le 1$;
		\item The Rips complex~$R_r(X_i)$ is contractible.
	\end{enumerate}
	Then~$R_r(\bigcup_{A_*} X_*)$ is contractible.
	\begin{proof}
		Let~$Y_0\coloneqq X_0$ and~$Y_{i+1}\coloneqq Y_i\cup_{A_i} X_{i+1}$ as in Definition~\ref{defn:succ_gluing}.
		In order to show that~$R_r(\bigcup_{A_*} X_*)$ is contractible, by Lemma~\ref{lem:exhaustion} it suffices to show for every~$i\in \IN$ that~$R_r(Y_i)$ is contractible.
		We proceed by induction on~$i$.
		The Rips complex~$R_r(Y_0)$ is contractible by assumption~(ii).
		By assumption~(i), Corollary~\ref{cor:gluing_star} applies to the metric gluing~$Y_{i+1}=Y_i\cup_{A_i} X_{i+1}$ showing that~$R_r(Y_{i+1})$ is homotopy equivalent to~$R_r(Y_i)\cup_{R_r(A_i)} R_r(X_{i+1})$.
		The Rips complexes~$R_r(Y_i)$, $R_r(X_{i+1})$, and~$R_r(A_i)$ are contractible by induction hypothesis, assumption~(ii), and assumption~(i), respectively.
		Hence the simplicial gluing of Rips complexes~$R_r(Y_i)\cup_{R_r(A_i)} R_r(X_{i+1})$ is contractible and thus also~$R_r(Y_{i+1})$ is contractible.
	\end{proof}
\end{prop}

\section{Applications}

\subsection{Graphs of groups with finite edge groups}

We prove that amalgamated products and HNN-extensions over finite groups preserve the property type~$\sfR$ (Definition~\ref{defn:type_R}).
The proofs are inspired by~\cite{Jakus13}.
By abuse of notation, we will also use~$(G,S)$ to refer to the metric space~$(G,d_S)$, where~$d_S$ is the word metric associated to~$S$.

\begin{proof}[Proof of Theorem~\ref{thm:amalg}]
	We will apply Proposition~\ref{prop:succ_gluing_star} to the amalgamated product~$(A\ast_C B,S_A\cup S_B)$ which is a successive metric gluing of copies of~$(A,S_A)$ and~$(B,S_B)$ along copies of~$(C,S_C)$~\cite[p.~280]{Cannon}, see also~\cite[Section~6]{Jakus13}.
	Note that the metric space~$(C,S_C)$ is indeed an isometric subspace of~$(A,S_A)$ and~$(B,S_B)$ because~$(C,S_C)$ is a finite simplex, as~$C$ is finite and~$S_C$ is the set of all non-trivial elements of~$C$.
	For the same reason, assumption~(i) of Proposition~\ref{prop:succ_gluing_star} is satisfied.
	Since~$(A,S_A)$ is of type~$\sfR$, there exists~$r_A\in \IR_{\ge 0}$ such that for all~$r\ge r_A$, the Rips complex~$R_r(A,S_A)$ is contractible.
	Similarly, there exists~$r_B\in \IR_{\ge 0}$.
	Let~$r_0\coloneqq \max\{r_A,r_B,2\}$.
	Then, for every~$r\ge r_0$, assumption~(ii) of Proposition~\ref{prop:succ_gluing_star} is satisfied.
	Hence, Proposition~\ref{prop:succ_gluing_star} yields that for every~$r\ge r_0$, the Rips complex~$R_r(A\ast_C B,S_A\cup S_B)$ is contractible.
	Thus~$(A\ast_C B,S_A\cup S_B)$ is of type~$\sfR$.
\end{proof}

We record the special case of Theorem~\ref{thm:amalg} when~$C$ is the trivial group.

\begin{cor}
\label{cor:free_prod}
	Let~$(A,S_A)$ and~$(B,S_B)$ be marked groups of type~$\sfR$.
	Then the free product~$(A\ast B,S_A\cup S_B)$ is of type~$\sfR$.
\end{cor}

\begin{proof}[Proof of Theorem~\ref{thm:HNN}]
	The proof is analogous to the proof of Theorem~\ref{thm:amalg} using that the HNN-extension~$(A\ast_{C,\varphi},S_A\cup \{t\})$ is a successive metric gluing of copies of~$(A,S_A)$ and~$(C\times \IZ/2\IZ\langle t\rangle,S_C\cup \{t\})$ along copies of~$(C,S_C)$~\cite[p.~281]{Cannon}, see also~\cite[Section~7]{Jakus13}.
	Here the metric space~$C\times \IZ/2\IZ\langle t\rangle$ consists of two copies $C\times \{0\}$ and~$C\times \{t\}$ of~$C$, where~$(c,0)$ has distance~$1$ to~$(c,t)$ for every~$c\in C$.
\end{proof}

\subsection{Two-dimensional right-angled Artin groups}

We prove that 2-dimen\-sional right-angled Artin groups with the standard generating set are of type~$\sfR$.

A \emph{cube complex} is a space obtained by gluing cubes (of possibly different dimensions) along faces and it is equipped with the metric induced by the $\ell^2$-metric on every cube.
By Gromov's link criterion, a simply-connected cube complex is $\CAT$ if and only if all vertex links are flag simplicial complexes.
We will consider the vertex set of a cube complex equipped with the $\ell^1$-metric (i.e., the graph metric on the 1-skeleton).
A \emph{square complex} is a cube complex of dimension~$\le 2$.

\begin{prop}
\label{prop:CAT}
	Let~$V$ be the vertex set of a locally finite $\CAT$ square complex equipped with the $\ell^1$-metric.
	Then, for all~$r\in \IR_{\ge 2}$, the Rips complex~$R_r(V)$ is contractible.
	\begin{proof}
		Let~$X$ be a locally finite $\CAT$ square complex with vertex set~$V$.
		Then~$X$ can be exhausted by finite $\CAT$ square subcomplexes~$(X_i)_{i\in \IN}$.
		Indeed, fix a vertex~$v_0$ of~$X$ and let~$X_i$ be the combinatorial convex hull of all vertices at distance~$\le i$ to~$v_0$. 
		Since~$X_i$ is combinatorially convex, it is also convex with respect to the $\CAT$ metric~\cite[Theorem~2.13]{Haglund} and hence~$X_i$ is $\CAT$. 
		Moreover, $X_i$ is a finite square complex because~$X$ is locally finite.
		By Lemma~\ref{lem:exhaustion}, it suffices to show for every~$i\in \IN$ that~$R_r(X_i^{(0)})$ is contractible.
		
		We may assume that~$X$ is a finite $\CAT$ square complex with vertex set~$V$.
		We proceed by induction on the number of vertices.
		If~$X$ is a single vertex, then the claim is true.
		Since the $\CAT$ square complex~$X$ is finite, there exists a free vertex~$v$, i.e., the vertex~$v$ is contained in a unique maximal cube~$C$, see e.g.~\cite[Proposition~2.4.11]{Rowlands_thesis}.
		Let~$X[V\setminus \{v\}]$ denote the full square subcomplex of~$X$ on the vertex set~$V\setminus \{v\}$.
		Then~$X$ is the gluing of~$X[V\setminus \{v\}]$ and~$C$ along~$A$, where~$A$ is a vertex (if~$C$ is an edge) or a path of length~$2$ (if~$C$ is a square).
		Hence~$V$ is the metric gluing of~$V\setminus \{v\}$ and~$C^{(0)}$ along~$A^{(0)}$, where all spaces are equipped with the $\ell^1$-metric.
		For both possibilities of~$A$, the vertex set~$A^{(0)}$ satisfies the assumption of Corollary~\ref{cor:gluing_star} which yields that the inclusion
		\[
			R_r(V\setminus \{v\})\cup_{R_r(A^{(0)})} R_r(C^{(0)})\to R_r(V)
		\]
		is a homotopy equivalence.
		Since~$C^{(0)}$ and~$A^{(0)}$ have diameter~$\le 2$, their respective Rips complexes at scale~$r\ge 2$ are contractible.
		The square complex~$X[V\setminus \{v\}]$ is $\CAT$, as it is simply-connected and its vertex links are flag, and has strictly fewer vertices than~$X$.
		By induction hypothesis, the Rips complex~$R_r(V\setminus v)$ is contractible. 
		Thus~$R_r(V)$ is contractible.
	\end{proof}
\end{prop}

We encourage the reader to visualise the proof of Proposition~\ref{prop:CAT} when~$V$ is the vertex set of~$\IR$, the 4-regular tree~$T_4$, $\IR\times \IR$, or~$T_4\times \IR$ with the obvious square complex structures.

\begin{proof}[Proof of Theorem~\ref{thm:RAAG}]
	The Salvetti complex is a finite CW-model for~$B(A_\Gamma)$ of dimension equal to the clique number of~$\Gamma$, which is at most~$2$ because~$\Gamma$ has no triangles.
	The universal covering~$X$ of the Salvetti complex is a locally finite $\CAT$ square complex~\cite[Theorem~3.6]{Charney_survey}.
	The vertex set of~$X$ equipped with the $\ell^1$-metric can be identified with~$A_\Gamma$ equipped with the word metric associated to~$\Gamma^{(0)}$.
	Hence the claim follows from Proposition~\ref{prop:CAT}.
\end{proof}

We record the special case of Theorem~\ref{thm:RAAG} when~$\Gamma$ is a complete bipartite graph.

\begin{ex}
\label{ex:bipartite}
	Let~$F_k$ and~$F_l$ be finitely generated free groups of rank~$k$ and~$l$, respectively.
	Then the direct product~$F_k\times F_l$ with the standard generating set is of type~$\sfR$.
\end{ex}

Proposition~\ref{prop:CAT} applies also to the infinite dihedral group and the Klein bottle group with suitable finite generating sets.

\section{Products of free abelian groups and finite groups}

In this section, we observe that direct products of the form~$\IZ^n\times F$, where~$F$ is a finite group, are of type~$\sfR$.
This follows by inspecting Zaremsky's proof for~$\IZ^n$ via a Morse theoretic criterion for the contractibility of Rips complexes~\cite{Zaremsky26}. 

For a subset~$U$ of a metric space~$X$ and~$\alpha\in \IR_{\ge 0}$, we denote by~$Z(U,\alpha)$ the set of elements~$z\in X$ such that~$U$ is contained in the ball of radius~$\alpha$ centred at~$z$.
A graph is uniformly locally finite if the degree of all vertices is bounded by a uniform constant.

\begin{thm}[{\cite[Theorem~3.1]{Zaremsky26}}]
\label{thm:Morse}
	Let~$X$ be the vertex set of a uniformly locally finite graph and let~$r_0\in \IR_{\ge 0}$.
	Suppose that for all~$r>r_0$, there exists~$\alpha_r<r$ such that the following holds:
	For every subset~$U\subset X$ with~$\diam(U)=r$, there exists~$z_0\in Z(U,\alpha_r)$ such that for every~$z\in Z(U,\alpha_r)$, we have~$d(z_0,z)\le r$.
	Then, for every~$r\ge r_0$, the Rips complex~$R_r(X)$ is contractible. 
\end{thm}

Zaremsky showed that~$\IZ^n$ satisfies the assumption of Theorem~\ref{thm:Morse} in a strong quantitative way.

\begin{thm}[{\cite[Proof of Corollary~3.2]{Zaremsky26}}]
\label{thm:Morse_Zn}
	Let~$n\in \IN$ and let~$\IZ^n$ be equipped with the standard generating set.
	For~$r\in \IR_{\ge 0}$, let
	\[
		\beta_r\coloneqq rn/(n+1)+n/2.
	\]
	Let~$U\subset \IZ^n$ with~$\diam(U)=r$.
	Then there exists~$z_0\in Z(U,\beta_r)$ such that for every~$\beta\in \IR_{\ge 0}$ and every~$z\in Z(U,\beta)$, we have~$d(z_0,z)\le \beta+n/2$.
	
	In particular, $\IZ^n$ satisfies the assumption of Theorem~\ref{thm:Morse} for~$r_0=n^2+n-1$.
\end{thm}

In~\cite[Corollary~3.2]{Zaremsky26}, Theorem~\ref{thm:Morse_Zn} is proved for~$\beta=\beta_r$ but the same proof works for every~$\beta$.
The key is that~$z_0$ can be chosen close to a point in the convex hull of~$U$ in the ambient~$\IR^n$.
In the following, our main insight is that $r-\beta_r$ becomes arbitrarily large for large~$r$, which can be used to compensate for any constant.

\begin{proof}[Proof of Theorem~\ref{thm:prod_finite}]
	We first fix some notation.
	Let~$d_{\IZ^n}$, $d_F$, and~$d_{\IZ^n\times F}$ denote the respective word metrics on~$\IZ^n$, $F$, and~$\IZ^n\times F$.
	For~$g_1,g_2\in \IZ^n$ and~$f_1,f_2\in F$, we have
	\[
		d_{\IZ^n\times F}\bigl((g_1,f_1),(g_2,f_2)\bigr) = d_{\IZ^n}(g_1,g_2) + d_F(f_1,f_2).
	\]
	Let~$D$ be the diameter of~$F$ and let~$p\colon \IZ^n\times F\to \IZ^n$ be the projection onto the first factor.
	Then, for a subset~$U\subset \IZ^n\times F$, we have 
	\[
		\diam_{\IZ^n}(p(U))\le \diam_{\IZ^n\times F}(U) \le \diam_{\IZ^n}(p(U))+D.
	\]
	For~$r\in \IR_{\ge 0}$, let~$\beta_r\coloneqq rn/(n+1)+n/2$ as in Theorem~\ref{thm:Morse_Zn}.

	We verify the assumption of Theorem~\ref{thm:Morse}.
	Let~$r_0\in \IR_{\ge 0}$ be large enough such that~$\beta_{r_0}+n/2+2D\le r_0$.
	Then, for all~$r>r_0$, also~$\beta_r+n/2+2D\le r$.
	We set~$\alpha_r\coloneqq \beta_r+D$.
	Let~$r>r_0$ and let~$U\subset \IZ^n\times F$ be a subset with~$\diam_{\IZ^n\times F}(U)=r$.
	Then~$\overline{r}\coloneqq \diam_{\IZ^n}(p(U))\le r$.
	By Theorem~\ref{thm:Morse_Zn}, there exists~$z_0\in Z_{\IZ^n}(p(U),\beta_{\overline{r}})$ such that for every~$\beta\in \IR_{\ge 0}$ and every~$z\in Z_{\IZ^n}(p(U),\beta)$, we have~$d_{\IZ^n}(z_0,z)\le \beta+n/2$.
	Since~$z_0\in Z_{\IZ^n}(p(U),\beta_{\overline{r}})$, we have~$(z_0,e_F)\in Z_{\IZ^n\times F}(U,\beta_{\overline{r}}+D)$, where~$e_F\in F$ is the neutral element.
	Since~$\beta_r$ is monotonously increasing in~$r$, we have $\beta_{\overline{r}}+D\le \beta_r+D=\alpha_r$.
	Hence, $(z_0,e_F)\in Z_{\IZ^n\times F}(U,\alpha_r)$.
	
	Let~$(g,f)\in Z_{\IZ^n\times F}(U,\alpha_r)$.
	Then~$g\in Z_{\IZ^n}(p(U),\alpha_r)$.
	By the properties of~$z_0$, we have~$d_{\IZ^n\times F}(z_0,g)\le \alpha_r+n/2$.
	Then, $d_{\IZ^n\times F}((z_0,e_F),(g,f))\le \alpha_r+n/2+D\le r$.
	Hence Theorem~\ref{thm:Morse} applies and yields the claim.
\end{proof}

\bibliographystyle{alpha}
\bibliography{bib}

\setlength{\parindent}{0cm}

\end{document}